\documentclass[a4paper,11pt]{amsart}
\usepackage{amsfonts,amsmath,amssymb,amsthm}
\usepackage{latexsym}
\usepackage{eucal}
\usepackage[dvips]{graphics}
\usepackage{mathabx}
\usepackage[english]{babel}
\usepackage{epsfig}
\usepackage{enumerate}
\usepackage{float}

\usepackage[usenames]{color}
\usepackage[colorlinks=true, linkcolor=red, citecolor=blue]{hyperref}
\newcommand{\nwc}{\newcommand}
\nwc\eps{\varepsilon}

\newtheorem{theorem}{Theorem}[section]
\newtheorem{proposition}[theorem]{Proposition}
\newtheorem{lemma}[theorem]{Lemma}

\newtheorem{claim}{Claim}
\newtheorem{remark}{Remark}
\theoremstyle{remark}

\newtheorem{definition}{Definition}
\usepackage{mathrsfs}

\numberwithin{equation}{section}
\makeatletter
\@namedef{subjclassname@2010}{\textup{2020} Mathematics Subject Classification}
\makeatother

\begin{document}

\title[Well-posedness and stabilization of higher-order Boussinesq system]{
Well-Posedness and Long-Time Dynamics of a Water-Waves Model with Time-Varying Boundary Delay}

\author[Bautista]{G. J. Bautista}
\address{Facultad de Ingeniería, Escuela Profesional  de Ingenier\'ia Civil, Universidad Tecnológica de los Andes (UTEA), Sede Abancay,
 Av. Perú 700, Apurímac, Peru.}
\email{gbautistas@utea.edu.pe}


\author[Capistrano-Filho]{R. de A.  Capistrano--Filho$^*$}
\address{
Departamento de Matem\'atica, Universidade Federal de Pernambuco\\
Cidade Universit\'aria, 50740-545, Recife (PE), Brazil\\
Email address: \normalfont\texttt{roberto.capistranofilho@ufpe.br}}
	\thanks{$^*$Corresponding author: roberto.capistranofilho@ufpe.br}

\author[Chentouf]{B. Chentouf}
\address{
Faculty of Science, Kuwait University \\
Department of Mathematics, Safat 13060, Kuwait\\
Email address: \normalfont\texttt{boumediene.chentouf@ku.edu.kw}}

\author[Sierra]{O. Sierra Fonseca}
\address{Escola de Matemática Aplicada\\ Fundação Getúlio Vargas, Praia de Botafogo 190, 22250-900, Rio de Janeiro (RJ), Brazil \\
Email address: \normalfont\texttt{oscar.fonseca@fgv.br}}



\date{\today}

\numberwithin{equation}{section}

\begin{abstract}
A higher-order nonlinear Boussinesq system with a time-dependent boundary delay is considered. Sufficient conditions are presented to ensure the well-posedness of the problem by utilizing Kato’s variable norm technique and the Fixed-Point Theorem. More significantly, the energy decay for the linearized problem is demonstrated using the energy method.
\end{abstract}

\maketitle

\section{Introduction}

\subsection{Background}
The Boussinesq system comprises a set of nonlinear partial differential equations (PDEs) that model wave dynamics in fluids with small amplitude and long wavelengths. Originally formulated by the French mathematician Joseph Boussinesq in the 19th century to describe shallow water waves \cite{bou}, this system has since been recognized as a model for various physical phenomena, including ocean currents, atmospheric circulation, and heat transfer in fluids. Consequently, the Boussinesq system remains an essential tool in fluid dynamics, with broad applications in fields such as meteorology, oceanography, and engineering.

In more recent studies, Bona \textit{et al.} \cite{Bona2002, Bona2004}  introduced a four-parameter family of Boussinesq systems to describe the motion of small-amplitude, long waves on the surface of an ideal fluid under gravity, particularly in scenarios where the motion is predominantly two-dimensional. In particular, the authors in \cite{Bona2002, Bona2004} investigated the following system:
\begin{align}\label{eq:Boussi}
\left\{\begin{array}{ll}
			\vspace{2mm}\eta_t +  \omega_{x} + a \omega_{xxx}-b\eta_{txx} + a_1 \omega_{xxxxx} + b_1 \eta_{txxxx} \\
			\vspace{2mm}\qquad\qquad\qquad\qquad\qquad=-(\eta \omega)_x +b (\eta \omega)_{xxx}- \alpha'(\eta \omega_{xx})_x, \\
			\vspace{2mm}\omega_t+\eta_{x} + c\eta_{xxx}-d\omega_{txx} + c_1\eta_{xxxxx}+d_1 \omega_{txxxx}\\
			\qquad\qquad\qquad\qquad\qquad=- \omega \omega_x -c(\omega \omega_x)_{xx}- (\eta \eta_{xx})_x +\beta' \omega_x \omega_{xx}+\rho\, \omega \omega_{xxx}.
		\end{array}\right.
	\end{align}
In this context, \(\eta\) represents the elevation of the fluid surface from its equilibrium position, while \(\omega = \omega_{\theta}\) denotes the horizontal velocity of the flow at a height \(\theta h\), where \(h\) is the undisturbed depth of the fluid and \(\theta\) is a constant within the interval \([0,1]\). The variables \(x\) and \(t\) correspond to space and time, respectively, and the physical parameters \(a, b, c, d, a_1, c_1, b_1, d_1\) must satisfy the following relationships:
	\begin{align*}
	\begin{cases}
		\begin{array}{ll}
			\vspace{2mm} a + b=\frac{1}{2}(\theta^2 - \frac{1}{3}), \quad c+d=\frac{1}{2}(1-\theta^2 ),\\
			\vspace{2mm}a_1 - b_1 = -\frac{1}{2}(\theta^2 - \frac{1}{3})b + \frac{5}{24}(\theta^2 - \frac{1}{5})^2,\\
			\vspace{2mm} c_1-d_1 = \frac{1}{2}(1-\theta^2)c + \frac{5}{24}(1- \theta^2 )(\theta^2-\frac{1}{5}),\\
 \alpha'=a+b-\frac{1}{3},\,\, \beta'=c+d -1,\,\, \rho=c+d.
		\end{array}
		\end{cases}
	\end{align*}

Stabilization results for the higher-order system \eqref{eq:Boussi} on the periodic domain were established in \cite{bautista-pazoto} under the conditions \( a_1 = c_1 = 0 \), with general damping applied to each equation. Furthermore, the local exact controllability of the system \eqref{eq:Boussi} was investigated in \cite{bautista-micu-pazoto}, where the control is localized within the interior of the domain and influences only one equation.

Negative controllability results are explored in \cite{BP,fp-comun} when the third and fifth order Korteweg-de Vries (KdV) terms are removed from the system mentioned above, that is, \eqref{eq:Boussi} with \(a = a_1 = c = c_1 = 0\). In this case, the system consists of two coupled Benjamin-Bona-Mahony (BBM)-type equations. The authors demonstrated that the linear model is approximately controllable but not spectrally controllable. This implies that although any state can be brought arbitrarily close to another, no finite linear combination of eigenfunctions, other than zero, can be driven to zero.

Let us now consider $b = d = b_1 = d_1= 0$ and make a scaling argument to obtain the fifth-order Boussinesq system
\begin{equation}\label{eq:KdV-KdVa}
\left\{\begin{array}{ll}
			\vspace{2mm}\eta_t +  \omega_{x} + a \omega_{xxx} + a_1 \omega_{xxxxx} =-(\eta \omega)_x - \alpha'(\eta \omega_{xx})_x, \\
			\omega_t+\eta_{x} + c\eta_{xxx} + c_1\eta_{xxxxx}=- \omega \omega_x -c(\omega \omega_x)_{xx}- (\eta \eta_{xx})_x +\beta' \omega_x \omega_{xx}+\rho \omega \omega_{xxx}.
		\end{array}\right.
\end{equation}
In the above system, we note that $c, a_1 \geq 0$. Thus, we consider the following case:
\begin{align}\label{conditions_para}
a=c>0, \, \, \, \text{and} \, \, \, c_1 =a_1 >0.
\end{align}

It is important to highlight that, to the best of our knowledge, there are no existing results that combine a damping mechanism with a boundary time-varying delay to achieve stabilization of the higher-order Boussinesq system associated with \eqref{eq:KdV-KdVa}. This gap forms the main motivation for the present work. It is also worth emphasizing the practical prevalence of time delays in control systems, a phenomenon that is virtually unavoidable due to factors such as the lag between sensors, actuators, and data processing. In light of this, significant efforts have been made to mitigate or eliminate the effects of constant time delays. For example, stability results under smallness conditions on the domain length $L$ and the initial data have been established in \cite{Valein2019, Valein2022} for the KdV equation and in \cite{ZAMP} for the Kawahara equation. In the case of time-dependent delays, similar results have been obtained in \cite{Parada2022} for the KdV equation and \cite{Boumediene2024} for the coupled KdV-KdV system.

\subsection{Notations and main results}
This article is concerned with the following system
\begin{equation}\label{eq:KdV-KdV}
\begin{cases}
\eta_t +  \omega_{x} + a \omega_{xxx} + a_1 \omega_{xxxxx} =-(\eta \omega)_x - \alpha'(\eta \omega_{xx})_x,& \text{in} \; \mathbb{R}^{+}\times(0,L),\vspace{0.2cm} \\
\omega_t+\eta_{x} + c\eta_{xxx} + c_1\eta_{xxxxx}=- \omega \omega_x -c(\omega \omega_x)_{xx}\vspace{0.2cm}\\ \quad\quad\quad\quad\quad\quad\quad\quad\quad\quad\quad\quad\quad\vspace{0.2cm}- (\eta \eta_{xx})_x +\beta' \omega_x \omega_{xx}+\rho \omega \vspace{2mm}\omega_{xxx}, &\text{in} \;  \mathbb{R}^{+}\times(0,L), \\
\vspace{2mm}\eta(t,0) = \eta(t,L) = \eta_x(t,0) = \eta_x(t,L) = \eta_{xx}(t,0)= 0, & t\in\mathbb{R}^{+},\\
\vspace{2mm}\omega(t,0) = \omega(t,L) = \omega_x(t,0) = \omega_x(t,L) = 0, & t\in\mathbb{R}^{+},\\
\vspace{2mm}\omega_{xx}(t,L) = \alpha \eta_{xx}(t,L)+\beta \eta_{xx}(t-\tau(t),L), & t>0,\\
\vspace{2mm}\eta_{xx}(t-\tau(0),L) = z_0(t-\tau(0))\in L^2(0,1) , & 0<t<\tau(0), \\
\left(\eta(0,x),\omega(0,x)\right) = \left(\eta_0(x),\omega_0(x)\right)\in X_0, & x\in(0,L),
\end{cases}
\end{equation}where  the parameters $a, c, a_1, c_1$
verify \eqref{conditions_para}. Moreover, we assume that there exist  positive constants $\tau_0$, $M$ and $d<1$ such that the time-dependent delay function $\tau(t)$ satisfies the following standard conditions:
\begin{equation}\label{eq:TauCond}
\begin{cases}
0 < \tau_0 \leq \tau(t) \leq M, \quad
\dot\tau(t)\leq d < 1,&\forall t\geq 0,\\
\tau\in W^{2, \infty}([0,T]),&T>0.
\end{cases}
\end{equation}

Finally, the feedback gains $\alpha >0$ and $\beta$ must obey the following constraint: 
\begin{equation}\label{matrix_nega_def_albe}
\text{The matrix} \; \Phi_{\alpha,\beta} = \begin{pmatrix}
-2a_1 \alpha + \lvert\beta\rvert & -a_1\beta\\
-a_1\beta & \lvert\beta\rvert (d -1 )
\end{pmatrix}
\text{is negative definite}. 
\end{equation}
The condition \eqref{matrix_nega_def_albe} ensures the dissipation of the system \eqref{eq:KdV-KdV}. It is worth mentioning that a similar condition is used for other types of delayed dispersive systems (see, for instance  \cite{Boumediene2024} and \cite{Parada2022}). Furthermore, recalling that $0\leq d <1$ and $a_1>0$, one can readily check that \eqref{matrix_nega_def_albe} is fulfilled if, for instance, the feedback gains $\alpha$ and $\beta$ satisfy
\begin{equation*}
\alpha >\frac{\lvert\beta\rvert}{2a_1}\left(\frac{a_1^2+1-d}{1-d}\right).
\end{equation*}

Next, let $X_{0}:= L^2(0,L)\times L^2(0,L),$ and the state space
$$ H:=X_0 \times L^2(0,1)$$
equipped with the inner product
\begin{equation*}
\begin{aligned}
\left\langle \left(\eta,\omega,z\right), \left(\tilde\eta, \tilde\omega, \tilde{z}\right)
\right\rangle_{t}
=\ &
\left\langle\left(\eta, \omega \right),\left(	\tilde\eta , \tilde\omega\right)
\right\rangle_{X_0}
+ \lvert \beta\rvert  \tau(t)
\left\langle
	z, \tilde z
\right\rangle_{L^2(0,1)} ,
\end{aligned}
\end{equation*}
for any $(\eta,\omega;z), (\tilde\eta,\tilde\omega;\tilde z)\in H$. Moreover, we shall consider the space
$$\mathcal{B}:=C([0, T ], X_0) \cap L^2 (0, T, [H^2_0(0,L)]^2),$$
whose norm is
$$
\lVert (\eta,\omega) \rVert_{\mathcal{B}} = \sup_{t\in[0,T]} \lVert (\eta(t),\omega(t)) \rVert_{X_0}
+\lVert (\eta,\omega) \rVert_{ L^2 (0, T, [H^2_0(0,L)]^2)}.$$

To present our first result, let us introduce the following space
\begin{equation}\label{x3}
X_3:=\left\{(\eta,\omega) \in \left[ H^3(0,L)\cap H_0^2(0,L) \right]^2 | \,\eta_{xx}(0)=0\right\}.
\end{equation}
The first result of this manuscript ensures the local well-posedness of the system \eqref{eq:KdV-KdV}.
\begin{theorem}\label{nonlinearT}
 Let $T>0$ and the parameters $a, c, a_1, c_1$ verify \eqref{conditions_para}. Then, there exists $\theta=\theta(T)>0$ such that, for every $\left(\eta_0, \omega_0;z_0\right) \in {X}_3\times L^2(0,1)$ satisfying
$$
\left\|\left(\eta_0, \omega_0\right)\right\|_{[ H^3(0,L)\cap H_0^2(0,L)]^2}<\theta,
$$
the system \eqref{eq:KdV-KdV} admits a unique solution $(\eta, \omega) \in C\left([0, T] ; X_3 \right)$. Moreover
$$
\|(\eta, \omega)\|_{C\left([0, T]: [ H^3(0,L)\cap H_0^2(0,L)]^2\right)} \leq C\left\|\left(\eta_0, u_0\right)\right\|_{[ H^3(0,L)\cap H_0^2(0,L) ]^2}
$$
for some positive constant $C=C(T)$.
\end{theorem}

Our second result is closely related to the total energy associated with the system \eqref{eq:KdV-KdV} that is defined in $H$ by
\begin{equation}\label{eq:En}
E(t) = \frac{1}{2} \int_0^L(  \eta^2(t,x) +  \omega^2(t,x) )\,dx
+ \frac{\lvert \beta\rvert }{2}\tau(t) \int_0^1 \eta_{xx}^2(t-\tau(t)\rho , L)\,d\rho.
\end{equation}
Indeed, the second result of the article guarantees that the energy $E(t)$ associated with the following system
\begin{equation}\label{eq:KdV-KdVaaaa}
\begin{cases}
\eta_t +  \omega_{x} + a \omega_{xxx} + a_1 \omega_{xxxxx} =0,& \text{in} \; \mathbb{R}^{+}\times(0,L), \\
\omega_t+\eta_{x} + c\eta_{xxx} + c_1\eta_{xxxxx}=-0, &\text{in} \;  \mathbb{R}^{+}\times(0,L), \\
\eta(t,0) = \eta(t,L) = \eta_x(t,0) = \eta_x(t,L) = \eta_{xx}(t,0)= 0, & t\in\mathbb{R}^{+},\\
\omega(t,0) = \omega(t,L) = \omega_x(t,0) = \omega_x(t,L) = 0, & t\in\mathbb{R}^{+},\\
\omega_{xx}(t,L) = \alpha \eta_{xx}(t,L)+\beta \eta_{xx}(t-\tau(t),L), & t>0,\\
\eta_{xx}(t-\tau(0),L) = z_0(t-\tau(0))\in L^2(0,1) , & 0<t<\tau(0), \\
\left(\eta(0,x),\omega(0,x)\right) = \left(\eta_0(x),\omega_0(x)\right)\in X_0, & x\in(0,L),
\end{cases}
\end{equation}
 decays exponentially, even in the presence of delay, and provides an estimate of the decay rate. The result is expressed as follows:
\begin{theorem}\label{th:Lyapunov0}
Let  the parameters $a, c, a_1, c_1$
verify \eqref{conditions_para} and  $0<L<\sqrt{\frac{5 a_1}{3 a}}\pi$. Suppose also that the time-dependent delay function satisfies \eqref{eq:TauCond}. Then, there exist two positive constants
\begin{equation}\label{eq:r}
\zeta = \frac{1 + \max\{\mu_1 L,\mu_2\}}{1- \max\{\mu_1 L,\mu_2\}},
\end{equation}
and
\begin{equation}\label{eq:lambda}
 \lambda \leq \min \left\lbrace
\frac{\mu_1 \pi^2\left(5 a_1\pi^2-3 a L^2\right)}{L^4(1+\mu_1L)}, \frac{\mu_2(1-d)}{M(1+\mu_2)}
\right\rbrace
\end{equation}
such that the energy $E(t)$ given by \eqref{eq:En} associated to the  system \eqref{eq:KdV-KdVaaaa} satisfies
$$
E(t) \leq \zeta E(0)e^{-\lambda t}, \quad \hbox{ for all } t \geq 0.$$
Here $\mu_1$ and $\mu_2 $ are two positive constants small enough to be well-chosen.
 \end{theorem}

The main contribution of this work is to establish the local well-posedness of system \eqref{eq:KdV-KdVa} and to prove the exponential stability of system \eqref{eq:KdV-KdVaaaa}. These results extend and refine those obtained in \cite{Boumediene2024} and \cite{Parada2022} in several significant directions. More specifically, unlike \cite{Boumediene2024}, where the nonlinear coupling appears only in one equation through the term $(\eta \omega)_x$, our model introduces an additional nonlinear coupling term of the form $(\eta \omega_{xx})_x$. Furthermore, while \cite{Boumediene2024} features a single uncoupled nonlinear term, our system includes four additional uncoupled nonlinear terms of higher order, making the analysis more intricate and requiring careful handling of the extra terms in the computations. In particular, in our case, the higher-order spatial derivatives (of order three and five) appear with positive signs, leading to conflicts between the $H^1$ and $H^2$-norm terms during integration by parts. Compared to \cite{Parada2022}, the situation is even more complex: the problem in \cite{Parada2022} involves a single uncoupled equation with only one nonlinearity and a highest-order derivative of three. Finally, and importantly, in contrast to \cite{Boumediene2024, Parada2022}, we employ the transposition method \cite{CF2019} to address the well-posedness of our system.

\subsection{Outline} The structure of the paper is as follows. In Section \ref{sec2}, we establish the well-posedness of the nonlinear problem \eqref{eq:KdV-KdV}, namely, we show Theorem \ref{nonlinearT} starting with an analysis of the linear system \eqref{eq:KdV-KdVaaaa} using the \textit{variable norm technique of Kato}, followed by the application of the Fixed-Point Theorem to prove well-posedness for the full nonlinear problem. Section \ref{sec3} focuses on the stability result presented in Theorem~\ref{th:Lyapunov0}, along with a discussion of the optimal decay rate. Finally, we conclude the paper with further remarks in Section \ref{sec4}.

\section{Well-posedness results}\label{sec2}
In the sequel, we will assume \( a = c > 0 \) and \( a_1 = c_1 > 0 \) in  \eqref{eq:KdV-KdV} as well as in \eqref{eq:KdV-KdVaaaa}. We will first examine the well-posedness of the linear system \eqref{eq:KdV-KdVaaaa} and subsequently analyze the properties of the nonlinear problem \eqref{eq:KdV-KdV} in suitable spaces.
\subsection{Linear problem}
Consider the following linear Cauchy problem
\begin{equation}\label{eq:Cauchy}
\begin{cases}
\dfrac{d}{dt} U(t) = A(t)U(t),& t>0,\\
 U(0) = U_0, & t>0,
\end{cases}
\end{equation}
where $A(t)\colon D(A(t))\subset {H}\to {H}$ is densely defined. If $D(A(t))$ is independent of time $t$, i.e., $D(A(t)) = D(A(0)),$ for $t > 0.$ The next theorem ensures the existence and uniqueness of the Cauchy problem \eqref{eq:Cauchy}.
\begin{theorem}[\cite{Kato1970}]\label{th:KatoCauchy}
Assume that:
\begin{enumerate}
\item $\mathcal{Z} = D(A(0))$ is a dense subset of $H$ and $ D(A(t)) = D(A(0))$, for all $t > 0$,
\item 
$A(t)$ generates a strongly continuous semigroup on $H$. Moreover, the family $\{A(t) \colon t\in [0, T ]\}$ is stable with stability constants $C, \ m$ independent of $t$.
\item $\partial_t A(t)$ belongs to $L_{\ast}^\infty([0, T ], B(\mathcal{Z}, H))$, the space of equivalent classes of essentially bounded, strongly measure functions from $[0, T ]$ into the set $B(\mathcal{Z}, H)$ of bounded operators from $\mathcal{Z}$ into $H$.
\end{enumerate}
Then, problem~\eqref{eq:Cauchy} has a unique solution $U \in C([0, T ], \mathcal{Z}) \cap C^1 ([0, T ], H)$ for any initial datum in $\mathcal{Z}$.
\end{theorem}

The task ahead is to apply the above result to ensure the existence of solutions for the linear system \eqref{eq:KdV-KdVaaaa}. Arguing as in \cite{xyl} and \cite{Nicaise2006, Nicaise2009}, let us define the auxiliary variable
$$z(t,\rho) = \eta_{xx}(t- \tau(t)\rho,L),$$
which satisfies the transport equation:
\begin{equation}\label{eq:tr}
\begin{cases}
\tau(t)z_t(t,\rho) + (1-\dot\tau(t)\rho)z_\rho(t,\rho) = 0, &t>0, \rho \in(0,1), \\
z(t,0) = \eta_{xx}(t,L), \ z(0,\rho) = z_0(-\tau(0)\rho), & t>0, \  \rho \in (0,1).
\end{cases}
\end{equation}

Now, we pick up $U = (\eta,\omega; z)^{T}$ and consider the time-dependent operator
$$A(t)\colon D(A(t))\subset {H}\to {H}$$ given by
\begin{equation}\label{eq:A}
A(t)\left(\eta, \omega , z\right) :=\left(-\omega_x-a\omega_{xxx}-a_1 \omega_{xxxxx}, -\eta_x-a\eta_{xxx}-a_1 \eta_{xxxxx}, \dfrac{\dot\tau(t)\rho -1}{\tau(t)} z_\rho\right),
\end{equation}
with a domain defined by
\begin{equation}\label{DAA}
D(A(t))=
\left\lbrace
	\begin{aligned}
		(\eta,\omega,z) \in H; \, & (\eta,\omega)\in \left[ H^5(0,L)\cap H_0^2(0,L) \right]^2, \;z  \in H^1(0,1), \\
	&\eta_{xx}(0)=0, z(0) = \eta_{xx}(L),  \omega_{xx}(L) = \alpha \eta_{xx}(L) + \beta z(1)
	\end{aligned}
\right\rbrace.
\end{equation}
This allows us to write the problem \eqref{eq:KdV-KdVaaaa} in the abstract form  \eqref{eq:Cauchy} using \eqref{eq:tr}-\eqref{DAA}. Additionally, it is noteworthy that $D(A(t))$ is independent of time $t$ since $D(A(t)) = D(A(0))$.

Now, taking the triplet $\lbrace A, H, \mathcal{Z}\rbrace$, with $A = \left\lbrace A(t)\colon t\in[0,T] \right\rbrace$, for some $T>0$ fixed and $\mathcal{Z} = D(A(0))$, we can state and prove the well-posedness result of \eqref{eq:Cauchy} related to $\lbrace A, H, \mathcal{Z}\rbrace$.
\begin{theorem} \label{well-lin}
Let the parameters $a, c, a_1, c_1$
verify \eqref{conditions_para}. Assume that $\alpha$ and $\beta$ are real constants such that  \eqref{matrix_nega_def_albe} holds. Taking $U_0 \in H$,  there exists a unique solution $U \in C([0, +\infty), H)$ to~\eqref{eq:Cauchy} whose operator is defined by \eqref{eq:A}-\eqref{DAA}. Moreover, if $U_0  \in D(A(0))$, then $U \in C([0, +\infty), D(A(0)))\cap C^1([0, +\infty), H).$
\end{theorem}
\begin{proof}
The result will be proved in a standard way (see, for instance, \cite{Nicaise2009}). First, it is not difficult to see that $\mathcal{Z} = D(A(0))$ is a dense subset of $H$ and $D(A(t))=D(A(0))$, for all $t>0$. Thus, the requirement (1) of Theorem \ref{th:KatoCauchy} is fulfilled.

Concerning the condition (2) of Theorem~\ref{th:KatoCauchy}, let us note that simple integrations by parts together with the boundary conditions yield
\begin{equation*}
\left\langle A(t) U, U \right\rangle_t - \kappa(t) \left\langle U, U \right\rangle_t
\leq \frac{1}{2} \left(\eta_{xx}(t, L), \eta_{xx}(t-\tau(t),L) \right)
\Phi_{\alpha,\beta}
\left(	\eta_{xx}(t, L), \eta_{xx}(t-\tau(t),L) \right)^T,
\end{equation*}
where
\begin{equation*}
\kappa(t) = \frac{(\dot\tau(t)^2+1)^{\frac12}}{2\tau(t)}. 
\end{equation*}
Owing to \eqref{matrix_nega_def_albe}, it follows that
$$\left\langle A(t) U, U \right\rangle_t - \kappa(t) \left\langle U,U \right\rangle_t \leq 0.$$
Consequently, $\tilde{A}(t) = A(t) - \kappa(t)I$ is dissipative.

Now, we claim the following:
\begin{claim}\label{CL1}
For all $t\in[0,T]$, the operator $A(t)$ is maximal, or equivalently, we have that $\lambda I - A(t)$ is surjective,  for some  $\lambda > 0$.
\end{claim}

In fact, let us fix $t\in[0,T]$. Given $(f_1,f_2, h)^T \in H$, we seek a solution $U = (\eta,\omega, z )^T \in D(A(t))$  of the equation $(\lambda I - A(t))U = (f_1, f_2, h)$, that is,
\begin{equation}\label{eq:WP_ewz}
\begin{cases}
\lambda\eta + \omega_x + a\omega_{xxx} + a_1\omega_{xxxxx} = f_1, \\
\lambda\omega + \eta_x + a\eta_{xxx}+ a_1\eta_{xxxxx} = f_2, \\
\lambda z + \left(\dfrac{1-\dot\tau(t)\rho}{\tau(t)}\right) z_\rho = h,\\
\eta(0) = \eta(L) = \eta_x(0) = \eta_x(L) = \eta_{xx}(0)=0, \\
\omega(0) = \omega(L) = \omega_x(0) = \omega_x(L) =0, \\
\omega_{xx}(L) = \alpha \eta_{xx}(L) + \beta z(1), z(0) = \eta_{xx}(L).
\end{cases}
\end{equation}
One can readily verify that $z$ is given by
\begin{equation*}
z(\rho) =
\begin{cases}\displaystyle
\eta_{xx}(L) e^{-\lambda\tau(t)\rho} +\displaystyle \tau(t) e^{-\lambda\tau(t)\rho} \int_0^\rho e^{\lambda\tau(t)\sigma} h (\sigma)\,d\sigma, &\!\!\!\text{if }\dot\tau(t) = 0, \\[1mm]
\\
\displaystyle
e^{\lambda\frac{\tau(t)}{\dot\tau(t)}\ln(1-\dot\tau(t)\rho)} \left[ \eta_{xx}(L)
\displaystyle+
\int_0^\rho \frac{h(\sigma)\tau(t)}{1-\dot\tau(t)\sigma}  e^{-\lambda\frac{\tau(t)}{\dot\tau(t)}\ln(1-\dot\tau(t)\sigma)}\,d\sigma \right], &\!\!\! \text{if }\dot\tau(t) \neq 0.
\end{cases}
\end{equation*}
Thereby, $z(1) = \eta_{xx}(L) g_0(t) + g_h(t)$, in which
\begin{equation*}
g_0(t) =
\begin{cases}\displaystyle
e^{-\lambda\tau(t)}, &\text { if } \dot\tau(t) = 0,\\ \displaystyle
e^{\lambda\frac{\tau(t)}{\dot\tau(t)}\ln(1-\dot\tau(t))}, & \text { if }\dot\tau(t) \neq 0,
\end{cases}
\end{equation*}
and
\begin{equation*}
g_h(t)=
\begin{cases}\displaystyle
 \tau(t) e^{-\lambda \tau(t)} \int_0^1 e^{\lambda \tau(t) \sigma} h(\sigma) d \sigma, & \text { if } \dot{\tau}(t)=0, \\
 \\
 \displaystyle
e^{\lambda \frac{\tau(t)}{\dot\tau(t)} \ln(1-\dot{\tau}(t))} \int_0^1 \frac{h(\sigma) \tau(t)}{1-\dot{\tau}(t) \sigma} e^{-\lambda \frac{\tau(t)}{\dot{\tau}(t)} \ln (1-\dot{\tau}(t) \sigma)} d \sigma, & \text { if } \dot{\tau}(t) \neq 0.
\end{cases}
\end{equation*}
Combining the latter with \eqref{eq:WP_ewz}, it follows that $\eta$ and $\omega$ are solutions of the system
\begin{equation}\label{eq:WP_ew1}
\begin{cases}
\lambda\eta + \omega_x +  a\omega_{xxx} + a_1\omega_{xxxxx} = f_1, \\
\lambda\omega + \eta_x + a\eta_{xxx}+ a_1\eta_{xxxxx} = f_2,
\end{cases}
\end{equation}
and satisfy the boundary conditions
\begin{equation*}
\begin{cases}
\eta(0) = \eta(L) = \eta_x(0) = \eta_x(L) = \eta_{xx}(0)=0, \\
\omega(0) = \omega(L) = \omega_x(0) = \omega_x(L) =0,\\
\omega_{xx}(L) = (\alpha+\beta g_0(t)) \eta_{xx}(L) + \beta g_h(t).
\end{cases}
\end{equation*}
Now, let $ \phi_1 \in C^{\infty}(\left[ 0, L \right])  $ be a function such that $ \phi_1(0)=\phi_{1}(L)=\phi_{1,x}(0)=\phi_{1,x}(L)=0 $ and $ \phi_{1,xx}(L)= 1$. Next, we define a function $\psi(x,\cdot) = \phi_1(x)\beta g_h(\cdot) \in C^\infty([0,L])$ and let $\hat{\omega}:=\omega - \psi$. This, together with \eqref{eq:WP_ew1}, implies that  $\eta$ and $\omega$ satisfy
\begin{equation*}
\begin{cases}
\lambda\eta + \omega_x +a \omega_{xxx}+a_1\omega_{xxxxx} = f_1-\left(\psi_x+\psi_{xxx}+\psi_{xxxxx}\right) =: \tilde{f_1}, \\
\lambda\omega + \eta_x + a\eta_{xxx}+a_1\eta_{xxxxx} = f_2- \lambda\psi =: \tilde{f_2},
\end{cases}
\end{equation*}
as well as the boundary conditions
\begin{equation*}
\begin{cases}
\eta(0) = \eta(L) = \eta_x(0) = \eta_x(L) = \eta_{xx}(0)=0, \\
\omega(0) = \omega(L) = \omega_x(0) = \omega_x(L) =0,\\
\omega_{xx}(L) = (\alpha+\beta g_0(t)) \eta_{xx}(L).
\end{cases}
\end{equation*}

Let us mention that for the sake of simplicity, we still use $\omega$ after translation. Then, we can verify that $0<g_0(t)<1$ (see, for instance, \cite{Boumediene2024}). Thus, thanks to \eqref {matrix_nega_def_albe}, we deduce  that  $\tilde\alpha := \alpha + \beta g_0(t) > 0$. Consequently, showing the Claim \ref{CL1} is equivalent to proving that $\lambda I - \hat{A}$ is surjective, where $\hat{A}$ is given by
$$
\hat{A} (\eta, \omega) = (-\omega_x-a\omega_{xxx}-a_1 \omega_{xxxxx}, -\eta_x-a\eta_{xxx}-a_1 \eta_{xxxxx}),
$$
with a dense domain
$$
D(\hat{A}) := \left\lbrace
(\eta,\omega)\in \left[H^5(0,L)\cap H_0^2(0,L)\right]^2
\colon
\eta_{xx}(0) = 0,\ \omega_{xx}(L) = \tilde\alpha \eta_{xx}(L)
\right\rbrace \subset X_0.
$$
Now, observe that adjoint of $\hat{A}$, denoted by $\hat{A}^\ast$, is defined by
$$
\hat{A}^\ast (u, v) = (u_x+a u_{xxx}+a_1 u_{xxxxx}, v_x+av_{xxx}+a_1 v_{xxxxx}),
$$with
$$
D(\hat{A}^\ast) := \left\lbrace
(u,v)\in \left[H^5(0,L)\cap H_0^2(0,L)\right]^2
\colon
v_{xx}(0) = 0,\ u_{xx}(L) = -\tilde\alpha v_{xx}(L)
\right\rbrace.
$$
Since
\begin{equation*}
\left\langle  \hat{A} (\eta, \omega),\left(\eta, \omega \right)
\right\rangle_{X_0} = -a_1 \tilde\alpha \eta^2_{xx}(L),
\end{equation*}
and
\begin{equation*}
 \left\langle  \hat{A}^\ast (u, v),\left(u, v  \right)
\right\rangle_{X_0} = -a_1 \tilde\alpha v^2_{xx}(L),
\end{equation*}
we can show that the operators $\hat{A}$ and $\hat{A}^\ast$ are dissipative. Therefore, the desired result follows from the Lummer-Phillips Theorem (see, for example, \cite{Pazy}). This shows the Claim \ref{CL1}. Consequently,  $\tilde{A}(t)$ generates a strongly semigroup on $H$ and $\tilde{A} = \{\tilde{A}(t), t \in [0, T ]\}$ is a stable family of generators in $H$, whose stability constant is independent of $t$. Thus, the condition (2) of Theorem \ref{th:KatoCauchy} is satisfied.

Lastly, since $\tau \in W^{2,\infty}([0, T ])$ for all $T>0$, we reach that
$$
\dot{\kappa}(t)=\frac{\ddot{\tau}(t) \dot{\tau}(t)}{2 \tau(t)\left(\dot{\tau}(t)^2+1\right)^{1 / 2}}-\frac{\dot{\tau}(t)\left(\dot{\tau}(t)^2+1\right)^{1 / 2}}{2 \tau(t)^2}
$$
is bounded on $[0, T]$ for all $T>0$ and
\begin{equation*}
\frac{d}{dt} A(t) U=\left(0, 0 ,\frac{\ddot{\tau}(t) \tau(t) \rho-\dot{\tau}(t)(\dot{\tau}(t) \rho-1)}{\tau(t)^2} z_\rho\right).
\end{equation*}
Moreover, the coefficient of $z_\rho$ is bounded on $[0, T]$, and the regularity (3) of Theorem~\ref{th:KatoCauchy} is satisfied.

To sum up, we verified the assumptions of Theorem~\ref{th:KatoCauchy} and hence  for each $U_0 \in D({A}(0))$, the Cauchy problem
\begin{equation*}
\begin{cases}
\tilde{U}_t(t)=\tilde{A}(t) \tilde{U}(t), &t>0,\\
 \tilde{U}(0)=U_0,
\end{cases}
\end{equation*}
has a unique solution $\tilde{U} \in C([0, \infty), H)$ and $\tilde{U} \in C([0, \infty), D({A}(0))) \cap C^1([0, \infty), H)$. Thus, the solution of~\eqref{eq:Cauchy} is explicitly given by $U(t)=e^{\int_0^t \kappa(s) d s} \tilde{U}(t)$.
\end{proof}

We also have the following result.
\begin{proposition}\label{pr:Diss}
Let the parameters $a, c, a_1, c_1$
verify \eqref{conditions_para}. Suppose $\alpha$ and $\beta$ are real constants such that~\eqref{matrix_nega_def_albe} holds. Then, for any mild solution of~\eqref{eq:Cauchy}, the energy $E(t)$ defined by~\eqref{eq:En} is non-increasing and
\begin{equation}\label{eq:EnDiss3}
\frac{d}{dt}E(t)
 =
 \frac{1}{2} \begin{pmatrix}
\eta_{xx}(t, L) \\ \eta_{xx}(t-\tau(t),L)
\end{pmatrix}^{T}
\Phi_{\alpha,\beta}
\begin{pmatrix}
\eta_{xx}(t, L) \\ \eta_{xx}(t-\tau(t),L)
\end{pmatrix},
\end{equation}where the matrix $\Phi_{\alpha,\beta}$ is given by
 \eqref{matrix_nega_def_albe}.
\end{proposition}

\begin{proof}
  The proof is straightforward and hence omitted.
\end{proof}

We now proceed to prove the Kato smoothing property, along with several \emph{a priori} estimates. These results are crucial for establishing the well-posedness of the system~\eqref{eq:KdV-KdV}. In the following, \((S_t(s))_{s \geq 0}\) represents the two-parameter semigroup of contractions associated with the operator \(A(t)\). We are now prepared to state the following result:

\begin{proposition}
	Let the parameters $a, c, a_1, c_1$
verify \eqref{conditions_para} and  $\alpha$ and $\beta$ are real constant  such that~\eqref{matrix_nega_def_albe} holds. Then, the following estimate holds:
	\begin{equation}\label{eq:Kato1}
		\lVert
		(\eta,\omega)
		\rVert_{X_0}^2
		+ \lvert \beta \rvert
		\lVert
		z
		\rVert_{L^2(0,1)}^2
		\leq
		\lVert
		(\eta_0,\omega_0)
		\rVert_{X_0}^2
		+ \lvert \beta \rvert
		\lVert
		z_0(-\tau(0)\cdot)
		\rVert_{L^2(0,1)}^2,
	\end{equation}
	Furthermore, for every initial condition $(\eta_0,\omega_0, z_0)\in {H}$, we have that
	\begin{equation}\label{eq:Katotr0}
		\lVert \eta_{xx}(\cdot,L) \rVert_{L^2(0,T)}^2
		+\lVert z(\cdot,1) \rVert_{L^2(0,T)}^2
		\leq  \lVert(\eta_0,\omega_0)\rVert_{X_0}^2 +
		\lVert z_0(-\tau(0)\cdot)\rVert_{L^2(0,1)}^2.
	\end{equation}
On the other hand, for the initial datum, we have the following estimates
	\begin{equation}\label{eq:Kato3}
		\begin{split}
		\lVert
		(\eta_0,\omega_0)
		\rVert_{X_0}^2
		\leq&
		\frac{1}{T}
		\lVert
		(\eta,\omega)
		\rVert_{L^2(0,T; X_0)}^2
		\\&+(2\alpha+\lvert\beta\rvert)
		\lVert
		\eta_{xx}(\cdot,L)
		\rVert_{L^2(0,T)}^2
		+\lvert\beta\rvert
		\lVert
		z(\cdot,1)
		\rVert_{L^2(0,1)}^2
		\end{split}
	\end{equation}
	and
	\begin{equation}\label{eq:Kato4}
		\lVert z_0(-\tau(0)\cdot)\rVert_{L^2(0,1)}^2 \leq
		C_1(d,M)\left(\lVert z(T,\cdot )\rVert_{L^2(0,1)}
		+  \lVert z(\cdot, 1) \rVert_{L^2(0,T)}^2\right).
	\end{equation}

 Finally, for $0<L<\sqrt{\frac{5 a_1}{3 a}}\pi$, the Kato smoothing effect  is verified
	\begin{equation}\label{eq:Kato2}
		\int_0^T\int_0^L
		\left(\eta_{xx}^2+\omega_{xx}^2\right)
		\,dx\,dt
		\leq
		C(L,T,\alpha)
		\left(
		\lVert (\eta_0,\omega_0) \rVert_{X_0}^2
		+
		\lVert z_0(-\tau(0)\cdot) \rVert_{L^2(0,1)}^2
		\right),
	\end{equation} and the map
	\begin{equation*}
		(\eta_0,\omega_0;z_0)\in{H} \mapsto (\eta,\omega;z) \in \mathcal{B} \times C(0,T; L^2(0,1))
	\end{equation*}
	is well-defined and continuous.
\end{proposition}

\begin{proof}
Using \eqref{eq:EnDiss3} and the fact that $\Phi_{\alpha,\beta}$ is  a symmetric negative definite matrix, we deduce the existence of a positive constant $C$, such that
 \begin{align*}
 E^{\prime}(t)
 =
 \frac{1}{2} \begin{pmatrix}
\eta_{xx}(t, L) \\ z(t, 1)
\end{pmatrix}^{T}
\Phi_{\alpha,\beta}
\begin{pmatrix}
\eta_{xx}(t, L) \\ z(t, 1)
\end{pmatrix} \leq -C\left(\eta_{xx}^2(t, L) + z^2(t, 1)\right).
 \end{align*}
 Thus, it follows from the above estimate that
 \begin{align} \label{se_limi_ener_2}
   E'(t) + \eta_{xx}^2(t,L) + z^2(t,1) \leq 0
 \end{align} Integrating \eqref{se_limi_ener_2} in $[0,s]$, for $0\leq s \leq T$, we get
	\begin{equation*}
		E(s) + \int_0^s \eta_{xx}^2(t,L)\,dt + \int_0^s z^2(t,1)\,dt \leq E(0),
	\end{equation*}
	and \eqref{eq:Kato1} is obtained. Taking $s=T$ and since $E(t)$ is a non-increasing function (see Proposition~\ref{pr:Diss}), the estimate~\eqref{eq:Katotr0} holds.

Secondly, the proof of estimates \eqref{eq:Kato3} and \eqref{eq:Kato4} is analogous to that of \cite{Boumediene2024}, and we will omit the details.

 Now, we show the inequality \eqref{eq:Kato2} provided that $0<L<\sqrt{\frac{5 a_1}{3 a}}\pi$. Initially,  multiplying the first equation of \eqref{eq:KdV-KdVaaaa} by $x\omega$ and the second one by $x\eta$. Next, adding the results, then integrating by parts over $(0,L)\times(0,T)$ and invoking \eqref{eq:Kato1}-\eqref{eq:Katotr0}, we obtain
\begin{equation}\label{mul_part_1}
 \begin{split}
 &\frac{1}{2}\int_0^T\int_0^L
			\left(\eta^2+\omega^2\right)
			\,dx\,dt  - \frac{3 a}{2}\int_0^T\int_0^L
			\left(\eta_x^2+\omega_x^2\right)
			\,dx\,dt + \frac{5 a_1}{2}\int_0^T\int_0^L
			\left(\eta_{xx}^2+\omega_{xx}^2\right)
			\,dx\,dt \\
    & =  \frac{a_1 L}{2}\int_0^T\left(\eta_{xx}^2(t,L)+\omega_{xx}^2(t,L)\right)dt - \int_0^L x\left( \eta(t,x)\omega(t,x)-\eta_0(x)\omega_0(x)\right)dx \\
    & \leq C(L,\alpha, a_1)\left( \lVert(\eta_0,\omega_0)\rVert_{X_0}^2
			+\lVert z_0(-\tau(0)\cdot) \rVert_{L^2(0,1)}^2\right),
 \end{split}
 \end{equation}
 for some positive constant $C(L,\alpha, a_1)$. Since $0<L<\sqrt{\frac{5 a_1}{3 a}}\pi$, from Poincaré inequality, there exists $C_L = \frac{1}{2}\left( 5a_1 \pi^2 - 3aL^2 \right)>0$, such that
\begin{equation}\label{mul_part_22}
\begin{split}
C_L\int_0^T\int_0^L
			\left(\eta_{xx}^2+\omega_{xx}^2\right)
			\,dx\,dt \leq&  - \frac{3 a}{2}\int_0^T\int_0^L
			\left(\eta_x^2+\omega_x^2\right)\,dx\,dt
			\\& + \frac{5 a_1}{2}\int_0^T\int_0^L
			\left(\eta_{xx}^2+\omega_{xx}^2\right)
			\,dx\,dt.
			\end{split}
\end{equation}
Thus, from \eqref{mul_part_1} and \eqref{mul_part_22}, we obtain  \eqref{eq:Kato2}.
\end{proof}

The next result ensures the existence of solutions to the fifth-order KdV-KdV system with sufficient regular source terms.

\begin{theorem}
	Suppose that \eqref{eq:TauCond} and \eqref{matrix_nega_def_albe} hold. Let $U_0 = (\eta_0, \omega_0 , z_0) \in H$ and the source terms $(f_1,f_2)\in L^1(0,T; X_0)$. Then, if  the parameters $a, a_1$
verify \eqref{conditions_para}, there exists a unique solution $U = (\eta,\omega, z) \in C([0,T], H)$ to
	\begin{equation*}
		\begin{cases}
			\eta_t +  \omega_{x} + a \omega_{xxx} + a_1 \omega_{xxxxx} =f_1, & t>0, x\in(0,L), \\
			\omega_t+\eta_{x} + a\eta_{xxx} + a_1\eta_{xxxxx} = f_2, & t>0, x\in(0,L), 
		\end{cases}
	\end{equation*}
	with boundary conditions as in \eqref{eq:KdV-KdV}.
	Moreover, for $T > 0$, there exists a positive constant $C$ such that the following estimates hold
	\begin{equation*}
\left\{	\begin{array}{l}
		\lVert (\eta,\omega; z) \rVert_{C([0,T], H)} \leq
		C (\lVert (\eta_0,\omega_0, z_0) \rVert_H + \lVert (f,g) \rVert_{L^1(0,T, X_0)} ),
\\[2mm]
		\lVert(\eta_{xx}(\cdot,L), z(\cdot,1))\rVert_{[L^2(0,T)]^{2}}^2
		\leq C(\lVert{(\eta_0,\omega_0,z_0)}\rVert_{H}^2 + \lVert (f,g) \rVert_{L^1(0,T, X_0)}^2),
		\end{array}
		\right.
	\end{equation*}
	and, for $0<L<\sqrt{\frac{5 a_1}{3 a}}\pi$,
 \begin{align*}
 \lVert (\eta,\omega) \rVert_{L^2(0,T, H^2_0(0,L))} \leq
		C (\lVert (\eta_0,\omega_0, z_0) \rVert_H + \lVert (f,g) \rVert_{L^1(0,T, X_0)}).
 \end{align*}
\end{theorem}
\begin{proof}
This proof is analogous to that of \cite[Theorem 2.5]{Boumediene2024}, and hence we omit it.
\end{proof}

\subsection{Nonlinear problem}
In this subsection, we show the well-posedness of the nonlinear problem \eqref{eq:KdV-KdV} by using the approach of \cite{CF2019}, where the solutions are obtained via the transposition method and the existence and uniqueness by using the Riesz-representation theorem.

To prove the well-posedness result for system \eqref{eq:KdV-KdV}, we consider the non-homogeneous system
\begin{equation}\label{KdV-KdV-h1h2}
\begin{cases}
\eta_t +  \omega_{x} + a \omega_{xxx} + a_1 \omega_{xxxxx} = h_1
,& \text { in }(0, T)\times(0, L),
\\
\omega_t+\eta_{x} + a\eta_{xxx} + a_1\eta_{xxxxx}= h_2 
,&\text { in }(0, T)\times(0, L), \\
\eta(t,0) = \eta(t,L) = \eta_x(t,0) = \eta_x(t,L) = \eta_{xx}(t,0)= 0, & t\in(0,T),\\
\omega(t,0) = \omega(t,L) = \omega_x(t,0) = \omega_x(t,L) = 0, & t\in(0,T),\\
\omega_{xx}(t,L) = f(t) 
& t\in(0,T),
\\
\left(\eta(0,x),\omega(0,x)\right) = \left(\eta_0(x),\omega_0(x)\right), & x\in(0,L),
\end{cases}
\end{equation}
where  the parameters $a, a_1$ verify \eqref{conditions_para}. Remember the definition of $X_3$ giving by \eqref{x3}, and also consider the following set
$$\tilde{X}_3:=\left\{(\varphi,\psi) \in \left[ H^3(0,L)\cap H_0^2(0,L) \right]^2 | \varphi_{xx}(0)=\psi_{xx}(L)=0\right\}.$$
We define a solution by transposition\footnote{See \cite{lions1,lions2} to justify the choice of the formula \eqref{transp}  below.} as follows.
\begin{definition}[Solution by transposition]
     Let $T>0,$ $(\eta_0,\omega_0)\in  X_3$, $f\in L^2(0,T)$ and
     $$(h_1,h_2)\in L^2(0,T,[H^{-2}(0,L)]^2).$$ A solution of the problem \eqref{KdV-KdV-h1h2} is a function $(\eta,\omega)\in C(0,T;X_3)$ such that, for all $\sigma\in [0,T]$ and $(\varphi_\sigma,\psi_\sigma)\in \tilde{X}_3$ the following identity holds
\begin{equation}\label{transp}
\begin{aligned}
\left\langle (\eta(\sigma), \omega(\sigma)),\left(\varphi_\sigma, \psi_\sigma\right)\right\rangle&_{[H^3(0,L)\cap H_0^2(0,L)]^2}= \left\langle\left(\eta_0, \omega_0\right),(\varphi(0), \psi(0))\right\rangle_{[H^3(0,L)\cap H_0^2(0,L)]^2} \\ &+\int_{0}^{\sigma} f(t)
\, \varphi_{x x}(t,L)dt  +\int_0^{\sigma}\left\langle\left(h_1(t), h_2(t)\right),(\varphi(t), \psi(t))\right\rangle_{(H^{-2},H^2_0)^2}dt,
\end{aligned}
\end{equation}
where the pair $(\varphi,\psi)$ is the solution of
\begin{equation}\label{adjoint}
\begin{cases}\varphi_t+\psi_x-a \psi_{x x x}+a_1 \psi_{x x xx x}=0, & \text { in }(0, L) \times(0, \sigma), \\ \psi_t+\varphi_x-a \varphi_{x x x}+a_1 \varphi_{x x x x x}=0, & \text { in }(0, L) \times(0, \sigma), \\ \varphi(0, t)=\varphi(L, t)=\varphi_x(0, t)=\varphi_x(L, t)=\varphi_{x x}(0, t)=0, & \text { on }(0, \sigma), \\ \psi(0, t)=\psi(L, t)=\psi_x(0, t)=\psi_x(L, t)=\psi_{x x}(L, t)=0, & \text { on }(0, \sigma), \\ \varphi(x, \sigma)=\varphi_\sigma, \quad \psi(x, \sigma)=\psi_\sigma, & \text { on }(0, L) .\end{cases}
\end{equation}
\end{definition}

Thanks to \cite[Corollary 2.5 and Proposition 2.6]{CF2019}, the following well-posedness result for the system \eqref{adjoint} is established:

\begin{proposition}
 For all $(\varphi_\sigma,\psi_\sigma)\in
\tilde{X}_3,$ system \eqref{adjoint} admits a unique solution
$(\varphi,\psi)\in C([0,\sigma];\tilde{X}_3)$  which satisfies
\begin{equation}\label{adjProp:1}
    \lVert (\varphi(t),\psi(t)) \rVert_{[H^3(0,L)\cap H_0^2(0,L)]^2} \leq
		C \lVert (\varphi_\sigma,\psi_\sigma) \rVert_{[H^3(0,L)\cap H_0^2(0,L)]^2}, \quad \forall t\in[0,\sigma].
\end{equation}
Additionally, we have that
\begin{equation}\label{adjProp:2}
\int_{0}^{\sigma}|\varphi_{xx}(L,t)|^2+|\psi_{xx}(0,t)|^2\,dt \leq C \lVert (\varphi_\sigma,\psi_\sigma) \rVert_{[H_0^2(0,L)]^2}^2.
\end{equation}
\end{proposition}

The following result gives us the existence and uniqueness of the solution for system \eqref{KdV-KdV-h1h2}.
\begin{lemma}\label{lemma:wellp}
     Let $T>0,\left(\eta_0, \omega_0\right) \in X_3,\left(h_1, h_2\right) \in L^2\left(0, T ; [H^{-2}(0,L)]^2\right)$ and $f \in L^2(0, T)$. There exists a unique solution $(\eta, \omega) \in C\left([0, T] ; X_3\right)$ of the system \eqref{KdV-KdV-h1h2}. Moreover, there exists a positive constant $C_T$, such that
\begin{equation}\label{Lemma:1}
\begin{split}
\|(\eta(\sigma), \omega(\sigma))\|_{[H^3(0,L)\cap H_0^2(0,L)]^2} \leq& C_T\left(\left\|\left(\eta_0, \omega_0\right)\right\|_{[H^3(0,L)\cap H_0^2(0,L)]^2}+\|f\|_{L^2(0,T)}\right. \\
&\left.+\left\|\left(h_1, h_2\right)\right\|_{L^2\left(0, T: [H_0^2(0,L)]^2\right)}\right),
\end{split}
\end{equation}
for all $\sigma \in[0, T]$.
\end{lemma}
\begin{proof}
   Let us define $\Delta$ as the linear functional given by the right-hand side of \eqref{transp}, that is
\begin{equation*}
\begin{aligned}
\Delta\left(\varphi_\sigma, \psi_\sigma\right)=& \left\langle\left(\eta_0, \omega_0\right),(\varphi(0), \psi(0))\right\rangle_{[H^3(0,L)\cap H_0^2(0,L)]^2}+\int_{0}^{\sigma} f(t)
\, \varphi_{x x}(t,L)dt \\ & +\int_0^{\sigma}\left\langle\left(h_1(t), h_2(t)\right),(\varphi(t), \psi(t))\right\rangle_{(H^{-2},H^2_0)^2}dt.
\end{aligned}
\end{equation*}
We infer from \eqref{adjProp:1}, \eqref{adjProp:2} and the Cauchy-Schwarz inequality that
\begin{equation*}
    \begin{split}
       |\Delta\left(\varphi_\sigma, \psi_\sigma\right)| \leq& \|(\eta_0, \omega_0)\|_{[H^3(0,L)\cap H_0^2(0,L)]^2}\|(\varphi_\sigma, \psi_\sigma)\|_{[H^3(0,L)\cap H_0^2(0,L)]^2}\\&  +\| f\|_{L^2(0,T)}
\| \varphi_{x x}(L)\|_{L^2(0,T)} \\ & +C\|(\varphi_\sigma, \psi_\sigma)\|_{[H^3(0,L)\cap H_0^2(0,L)]^2} \|(h_1, h_2)\|_{L^1(0,T:H^{-2}(0,L)}
\\  \leq& C_{T} \left(\|(\eta_0, \omega_0)\|_{[H^3(0,L)\cap H_0^2(0,L)]^2}\right.\\
&\left.  +\| f\|_{L^2(0,T)}+ \|(h_1, h_2)\|_{L^1(0,T:H^{-2}(0,L)}\right) \|(\varphi_\sigma, \psi_\sigma)\|_{[H^3(0,L)\cap H_0^2(0,L)]^2},
\end{split}
\end{equation*}
and we obtain that $\Delta \in \mathcal{L}([H^3(0,L)\cap H_0^2(0,L)]^2;\mathbb{R}). $ Thus, from the Riesz representation Theorem, there
exists one and only one $(\eta_\sigma,\omega_\sigma)\in[H^3(0,L)\cap H_0^2(0,L)]^2$ such that
\begin{equation}\label{Riesz}
    \begin{aligned}
    \begin{cases}
    &\Delta\left(\varphi_\sigma, \psi_\sigma\right)= \left\langle\left(\eta_\sigma,\omega_\sigma\right),(\varphi_\sigma, \psi_\sigma)\right\rangle_{[H^3(0,L)\cap H_0^2(0,L)]^2}
\\
&\text{with} \ \ \  \|\Delta\|_{\mathcal{L}([H^3(0,L)\cap H_0^2(0,L)]^2;\mathbb{R})}=\|\left(\eta_\sigma,\omega_\sigma\right)\|_{[H^3(0,L)\cap H_0^2(0,L)]^2}
\end{cases}
 \end{aligned}
\end{equation}
and we obtain the uniqueness of the solution to the problem \eqref{KdV-KdV-h1h2}. Now, in order to prove the estimate \eqref{Lemma:1}, we define the map $(\eta,\omega):[0,T]\rightarrow [H^3(0,L)\cap H_0^2(0,L)]^2 $ as
$$
(\eta(\sigma),\omega(\sigma)):=(\eta_\sigma,\omega_\sigma) \ \text{for all} \ \sigma \in [0,T].
$$
and from \eqref{Riesz} we conclude that
\begin{equation*}
\begin{aligned}
\|(\eta(\sigma), \omega(\sigma))\|_{[H^3(0,L)\cap H_0^2(0,L)]^2}=&\|\Delta\|_{\mathcal{L}([H^3(0,L)\cap H_0^2(0,L)]^2;\mathbb{R})} \\
\leq& C_T\left(\left\|\left(\eta_0, \omega_0\right)\right\|_{[H^3(0,L)\cap H_0^2(0,L)]^2}+\|f\|_{L^2(0,T)}\right. \\
& \left.+\left\|\left(h_1, h_2\right)\right\|_{L^2\left(0, T: [H_0^2(0,L)]^2\right)}\right).
\end{aligned}
\end{equation*}
Finally, the fact that $(\eta, \omega) \in C\left([0, T] ; X_3\right)$ was already proved in \cite{Capistrano2018, CF2019} and hence we omit the details.
\end{proof}
Now, we pass to show the well-posedness of the non-homogeneous feedback linear
system associated to \eqref{KdV-KdV-h1h2}

\begin{lemma}\label{lemma:f}
     Let $T>0$. Then, for every $\left(\eta_0, \omega_0\right)$ in ${X}_3$ and $\left(h_1, h_2\right)$ in $L^2\left(0, T ; [H^{-2}(0,L)]^2\right)$, there exists a unique solution $(\eta, \omega)$ of the system \eqref{KdV-KdV-h1h2} such that
$
(\eta, \omega) \in C\left([0, T] ; {X}_3\right),
$
with $f(t) =\alpha \eta_{xx}(t,L)+\beta \eta_{xx}(t-\tau(t),L)$, where $\alpha$ and $\beta$ belong to $\mathbb{R}$. Moreover, for some positive constant $C=C(T)$, we have
$$
\|(\eta(t), \omega(t))\|_{[H^3(0,L)\cap H_0^2(0,L)]^2} \leq C\left(\left\|\left(\eta_0, \omega_0\right)\right\|_{[H^3(0,L)\cap H_0^2(0,L)]^2}+\left\|\left(h_1, h_2\right)\right\|_{L^2\left(0, T ;[H^{-2}(0,L)]^2\right)}\right),
$$
for all $t\in [0,T].$
\end{lemma}
\begin{proof}
   Note that if $
(\eta, \omega) \in C\left([0, T] ; {X}_3\right),$ from the trace theorems, it follows that
$$
f(t) =\alpha \eta_{xx}(t,L)+\beta \eta_{xx}(t-\tau(t),L)\in L^2(0,T).
$$

We claim that:  there exists a positive constant $C_{\alpha,\beta}$ such that
\begin{equation}\label{ineq:f}
    \|f\|_{L^2(0,T)}\leq C_{\alpha,\beta}T^{1/2}\|(\eta, \omega)\|_{C([0,T];[H^3(0,L)\cap H_0^2(0,L)]^2)}. \end{equation}
    Indeed, note that
    \begin{equation*}
    \begin{aligned}
        \|f\|_{L^2(0,T)}^2&\leq |\alpha|^2 CT\|\eta\|_{C([0,T];H^3(0,L))}^2+|\beta|^2 \int_{0}^{T}|\eta(t-\tau(t),L)|^2dt \\
       & \leq |\alpha|^2 CT|\eta\|_{C([0,T];H^3(0,L))}^2+|\beta|^2 \int_{0}^{T-\tau(T)}|\eta(s,L)|^2 \frac{1}{1-\dot{\tau}(t)} ds.
        \end{aligned}
    \end{equation*}
    By using the conditions \eqref{eq:TauCond}, we deduce, for some positive constant $C_M,$ that
    \begin{equation*}
    \begin{aligned}
        \|f\|_{L^2(0,T)}^2  &\leq \left(|\alpha|^2 C+|\beta|^2 \frac{C_{M}}{1-d}\right)T\,\|\eta\|_{C([0,T];H^3(0,L))}^2,
        \end{aligned}
    \end{equation*}
    giving the claim.

 Now, let $0<\gamma\leq T$ be determined later. For each $(\eta_0,\omega_0)\in X_3,$ consider the map
\begin{equation*}
    \begin{aligned}
        \Gamma: C([0,\gamma];[H^3(0,L)\cap H_0^2(0,L)]^2)&\longrightarrow C([0,\gamma];[H^3(0,L)\cap H_0^2(0,L)]^2)
        \\
        (\eta,\omega) &\longmapsto \Gamma(\eta,\omega)=(w,v),
    \end{aligned}
\end{equation*}
where $(w,v)$ is the solution of \eqref{KdV-KdV-h1h2} with $f(t) =\alpha \eta_{xx}(t,L)+\beta \eta_{xx}(t-\tau(t),L).$ By Lemma \ref{lemma:wellp} and \eqref{ineq:f}, the linear operator $\Gamma$ is well defined.  Furthermore, there exists a positive
constant $C_\gamma$, such that
\begin{equation*}
    \begin{aligned}
    &\|\Gamma(\eta,\omega)\|_{C([0,\gamma];[H^3(0,L)\cap H_0^2(0,L)]^2)}  \leq   C_\gamma\left(\left\|\left(\eta_0, \omega_0\right)\right\|_{[H^3(0,L)\cap H_0^2(0,L)]^2} \right. \\
    &\ \ \ \ \ \ \left. +\|\alpha \eta_{xx}(L)+\beta \eta_{xx}(\cdot-\tau(\cdot),L)\|_{L^2(0,\gamma)} +\left\|\left(h_1, h_2\right)\right\|_{L^2\left(0,\gamma: [H^{-2}(0,L)]^2\right)}\right).
    \end{aligned}
\end{equation*}
From \eqref{ineq:f} it follows that
\begin{equation*}
    \begin{aligned}
    \|\Gamma(\eta,\omega)\|_{C([0,\gamma];[H^3(0,L)\cap H_0^2(0,L)]^2)} \leq&   C_\gamma\left(\left\|\left(\eta_0, \omega_0\right)\right\|_{[H^3(0,L)\cap H_0^2(0,L)]^2}  +\left\|\left(h_1, h_2\right)\right\|_{L^2\left(0,\gamma: [H^{-2}(0,L)]^2\right)}\right) \\
    &+ C_{\alpha,\beta}\gamma^{1/2}\|(\eta, \omega)\|_{C([0,T];[H^3(0,L)\cap H_0^2(0,L)]^2)}.  \end{aligned}
\end{equation*}
Let $(\eta,\omega)\in B_R(0)$ where
{\small $$ B_R(0):=\left\{(\eta,\omega)\in C\left([0,\gamma];[H^3(0,L)\cap H_0^2(0,L)]^2\right)\ : \|(\eta, \omega)\|_{C([0,\gamma];[H^3(0,L)\cap H_0^2(0,L)]^2)}\leq R \right\},$$}
and
$$R=2 C_T\left(\left\|\left(\eta_0, \omega_0\right)\right\|_{[H^3(0,L)\cap H_0^2(0,L)]^2}  +\left\|\left(h_1, h_2\right)\right\|_{L^2\left(0,T: [H^{-2}(0,L)]^2\right)}\right).$$
Choosing $\gamma$ such that $$C_{\alpha,\beta}\gamma^{1/2}\leq\frac{1}{2},$$ it follows that
$$
 \|\Gamma(\eta, \omega)\|_{C\left([0,\gamma];[H^3(0,L)\cap H_0^2(0,L)]^2\right)}\leq R
 $$
        and
  $$
        \|\Gamma(\eta_1, \omega_1)-\Gamma(\eta_2, \omega_2)\|_{C\left([0,\gamma];[H^3(0,L)\cap H_0^2(0,L)]^2\right)}\leq \frac{1}{2}\|(\eta_1, \omega_1)-(\eta_2, \omega_2)\|_{C\left([0,\gamma];[H^3(0,L)\cap H_0^2(0,L)]^2\right)}.
    $$
Hence, $\Gamma: B_R(0)\rightarrow B_R(0)$ is a contraction, By Banach fixed point theorem, we obtain a unique $(\eta,\omega)\in B_R(0),$ such that $\Gamma(\eta, \omega)=(\eta, \omega)$ and
\begin{equation*}
    \begin{split}
\|(\eta,\omega)\|_{C([0,\gamma];[H^3(0,L)\cap H_0^2(0,L)]^2)} \leq&   2C_T\left(\left\|\left(\eta_0, \omega_0\right)\right\|_{[H^3(0,L)\cap H_0^2(0,L)]^2}\right.\\&\left. +\left\|\left(h_1, h_2\right)\right\|_{L^2\left(0,T: [H^{-2}(0,L)]^2\right)}\right).
    \end{split}
\end{equation*}
Since $\gamma$ is independent of $(\eta_0,\omega_0)$ the standard continuation extension argument
yields that the solution $(\eta,\omega)$ belongs to $C([0,T];[H^3(0,L)\cap H_0^2(0,L)]^2),$ and the proof ends.
\end{proof}

The first main result of the article ensures the existence of local solutions to \eqref{eq:KdV-KdV} and is proved below.

\begin{proof}[Proof of Theorem \ref{nonlinearT}]
    Let $T>0$ and $\|(\eta_0, \omega_0)\|_{ [ H^3(0,L)\cap H_0^2(0,L)]^2}<\theta,$ where $\theta>0$ will be determined later. We know from \cite{CF2019} that for $(\eta, \omega)\in C\left([0, T];[ H^3(0,L)\cap H_0^2(0,L)]^2\right),$ there exists a positive constant $C_1,$ such that the following inequalities hold true
    $$\|\eta \omega_x\|_{L^2(0,T;L^2(0,L))}\leq C_{1}T^{1/2}\|(\eta,\omega)\|^2_{C(0,T;[H^3(0,L)\cap H^2_{0}(0,L)]^2)},$$
     $$\|\eta_x \omega_{xx}\|_{L^2(0,T;L^2(0,L))}\leq C_{1}T^{1/2}\|(\eta,\omega)\|^2_{C(0,T;[H^3(0,L)\cap H^2_{0}(0,L)]^2)}$$
     and
$$ \|\eta \omega_{xxx}\|_{L^2(0,T;L^2(0,L))} \leq C_{1}T^{1/2}\|(\eta,\omega)\|^2_{C(0,T;[H^3(0,L)\cap H^2_{0}(0,L)]^2)}. $$
Thus,  the nonlinearities
$$
(h_1,h_2):=\left(-(\eta \omega)_x - \alpha'(\eta \omega_{xx})_x,
- \omega \omega_x -c(\omega \omega_x)_{xx}- (\eta \eta_{xx})_x +\beta' \omega_x \omega_{xx}+\rho \omega \omega_{xxx}\right)
$$
belong to $L^2(0,T;[L^2(0,L)]^2),$ and
\begin{equation}\label{ineq:h1h2}
\begin{aligned}
\begin{cases}
    \|h_1\|_{L^2(0,T;L^2(0,L))} \leq (2+2|\alpha'|) C_{1}T^{1/2}\|(\eta,\omega)\|^2_{C(0,T;[H^3(0,L)\cap H^2_{0}(0,L)]^2)}, \\
    \|h_2\|_{L^2(0,T;L^2(0,L))} \leq (3+4|c|+|\beta'|+|\rho|) C_{1}T^{1/2}\|(\eta,\omega)\|^2_{C(0,T;[H^3(0,L)\cap H^2_{0}(0,L)]^2)}.
    \end{cases}
\end{aligned}
\end{equation}
Taking this into consideration, we define the following map
\begin{equation*}
    \begin{aligned}
        \Gamma: C([0,T];[H^3(0,L)\cap H_0^2(0,L)]^2)&\longrightarrow C([0,T];[H^3(0,L)\cap H_0^2(0,L)]^2)
        \\
        (\eta,\omega) &\longmapsto \Gamma(\eta,\omega)=(\bar{\eta},\bar{\omega}),
    \end{aligned}
\end{equation*}
where $(\bar{\eta},\bar{\omega})$ is the solution of \eqref{KdV-KdV-h1h2} with $$ (h_1,h_2)\in L^2(0,T;[L^2(0,L)]^2)\subset L^2(0,T;[H^{-2}(0,L)]^2)$$ as defined above, and with
$f(t) =\alpha \eta_{xx}(t,L)+\beta \eta_{xx}(t-\tau(t),L).$ From Lemma \ref{lemma:f} we find that $\Gamma$ is well defined and  there exists a positive constant $C_{T}$ such that
\begin{equation*}
\begin{aligned}
\|\Gamma(\eta, \omega)\|_{C([0,T]:[H^3(0,L)\cap H_0^2(0,L)]^2)} \leq& C_T\left(\left\|\left(\eta_0, \omega_0\right)\right\|_{[H^3(0,L)\cap H_0^2(0,L)]^2}+\left\|\left(h_1, h_2\right)\right\|_{L^2\left(0, T: [H_0^2(0,L)]^2\right)}\right).
\end{aligned}
\end{equation*}
On the other hand, from the inequalities \eqref{ineq:h1h2} we have that
\begin{equation}\label{ineq:Gamma}
\begin{aligned}
\|\Gamma(\eta, \omega)\|_{C([0,T]:[H^3(0,L)\cap H_0^2(0,L)]^2)} \leq &C_T\left\|\left(\eta_0, \omega_0\right)\right\|_{[H^3(0,L)\cap H_0^2(0,L)]^2} \\
&+C_T 13 K C_1T^{1/2}\|(\eta, \omega)\|_{C([0,T]:[H^3(0,L)\cap H_0^2(0,L)]^2)}^2,
\end{aligned}
\end{equation}
where $K=\max\{1,|c|,|\alpha'|,|\beta'|,|\rho|\}.$ Now, we  consider the ball
$$
B_{R}(0)=\{(\eta,\omega)\in C([0,T]:[H^3(0,L)\cap H_0^2(0,L)]^2) \, :\,   \|(\eta, \omega)\|_{C([0,T]:[H^3(0,L)\cap H_0^2(0,L)]^2)}\leq R\},
$$
with $$R=2C_T \|(\eta_0,\omega_0)\|_{[H^3(0,L)\cap H_0^2(0,L)]^2}.$$ The inequality \eqref{ineq:Gamma} leads to
$$
\|\Gamma(\eta, \omega)\|_{C([0,T]:[H^3(0,L)\cap H_0^2(0,L)]^2)} \leq  \frac{R}{2}+C_T 13 K C_1T^{1/2}R^2 \leq   \frac{R}{2}+C_T^2 26 K C_1T^{1/2}\theta R
$$
for all $(\eta,\omega)\in B_R(0).$ By choosing $\theta$ such that
$$
    C_T^2 26 K C_1T^{1/2}\theta < \frac{1}{4}
$$
we obtain that $\Gamma(B_{R}(0))\subset B_{R}(0).$ Finally, following the same argument as done in \ref{lemma:f}, we can conclude that $\Gamma$ is a contraction in $B_{R}(0)$, then, the Banach fixed-point theorem guarantees the existence of a unique  $(\eta,\omega)\in B_{R}(0)$ such that $\Gamma(\eta,\omega)=(\eta,\omega)$ and
$$
\|(\eta, \omega)\|_{C\left([0, T]: [ H^3(0,L)\cap H_0^2(0,L)]^2\right)} \leq 2C_T\left\|\left(\eta_0, u_0\right)\right\|_{[ H^3(0,L)\cap H_0^2(0,L) ]^2},
$$
achieving the proof.
\end{proof}

\begin{remark}
In \eqref{eq:TauCond}, the time-dependent delay is assumed to be positive for all $t\ge 0$. This requirement is relaxed in \cite{np3} since $\tau$ is allowed to degenerate. Notwithstanding, the problem in \cite{np3} is linear and hence simpler than ours. The key idea of the proof in \cite{np3} is to consider a new delay $\tau_{\epsilon}$ defined by
\[ \tau_{\epsilon}(t):=\epsilon+\tau_{\epsilon}(t),\]
where $0<{\epsilon}<\overline{\epsilon}$, for some $\overline{\epsilon}>0$. Therefore, $\tau_{\epsilon}$ satisfy \eqref{eq:TauCond} and hence the problem \eqref{eq:Cauchy} has a unique solution $U_{\epsilon}$. The whole task is to tend $\epsilon$ to $0$ under more regularity on the solution. We have tried to adopt this approach, but we faced difficulties because of the nonlinearities in our problem.
\end{remark}

\section{Long-time behavior of solutions}\label{sec3}
In this section, we are in a position to prove the second main result of our work. First, we demonstrate that the energy associated with \eqref{eq:KdV-KdVaaaa} is exponentially stable. Moreover, we establish that the solutions decay at an optimal rate.
\subsection{Proof of Theorem~\ref{th:Lyapunov0}}
Recall that Theorem \ref{well-lin} (see also Proposition \ref{pr:Diss}) guarantees the $L^2$ a priori estimate for the linear system \eqref{eq:Cauchy} whose operator is defined by \eqref{eq:A}-\eqref{DAA}. Therefore, the solutions of the system \eqref{eq:KdV-KdVaaaa} are globally well-posed. Whereupon, we can treat the exponential stabilization for this system.

To proceed, consider the following Lyapunov functional $$V(t) = E(t) -\mu_1V_1(t) + \mu_2 V_2(t),$$ where $\mu_1,\mu_2 \in \mathbb{R}^+$ will be chosen later. Here, $E(t)$ is the total energy given by \eqref{eq:En},   while
$$
V_1(t) = \int_0^L x\eta(t,x)\omega(t,x)\,dx $$
and $$ V_2(t) = \frac{\lvert\beta\rvert}{2} \tau(t) \int_0^1 (1-\rho) \eta_{xx}^2(t-\tau(t)\rho ,L)\,d\rho.$$
Observe that,
\begin{equation}\label{eq:EquivEV}
	(1-\max\lbrace \mu_1 L, \mu_2\rbrace) E(t) \leq V(t) \leq (1+\max\lbrace\mu_1 L,\mu_2\rbrace)E(t),
\end{equation} 	
by assuming $0<\mu_1 <1/L$ and $0<\mu_2 <1$.

On the other hand, using the system \eqref{eq:KdV-KdVaaaa} and the boundary conditions, we get that
\begin{equation}\label{de_V_1_im}
\begin{split}
  V_1'(t)=&\int_{0}^{L} x\eta_t\omega dx+\int_{0}^{L}x\eta \omega_t dx \\
=& -\frac{a_1 L}{2} \begin{pmatrix}
		\eta_{xx}(t,L) \\ \eta_{xx}(t-\tau(t),L)
	\end{pmatrix}^T
	\begin{pmatrix}
		\alpha^2+1 & \alpha\beta \\
		\alpha\beta & \beta^2
	\end{pmatrix}
	\begin{pmatrix}
		\eta_{xx}t,L) \\ \eta_{xx}(t-\tau(t),L)
	\end{pmatrix} \\
	&+\frac{1}{2} \int_{0}^{L}\left(\omega^2+\eta^2\right)dx-\frac{3 a}{2}\int_{0}^{L}\left(\omega_{x}^{2}+\eta_{x}^{2}\right)dx + \frac{5 a_1}{2}\int_{0}^{L}\left(\omega_{xx}^{2}+\eta_{xx}^{2}\right)dx.
\end{split}
\end{equation}
In addition, from \eqref{eq:tr} and by integration by parts, we deduce that
\begin{equation}\label{V_2'}
	V_2'(t)  = -\frac{\lvert\beta\rvert}{2}\int_0^1(1-\dot\tau(t)\rho)\eta_{xx}^2(t-\tau(t)\rho,L) d\rho + \frac{\lvert\beta\rvert}{2}\eta_{xx}^2(t,L).
\end{equation}
Thus, from \eqref{eq:EnDiss3}, \eqref{de_V_1_im} and \eqref{V_2'}, we have that
\begin{align}
V^{\prime}(t)+\lambda V(t) = S_1+S_2 + S_3,\label{DV}
\end{align}where
\begin{align*}
S_1=&	\frac{1}{2} \left\langle \Psi_{\mu_1,\mu_2} (\eta_{xx}(t,L), \eta_{xx}(t-\tau(t),L)), (\eta_{xx}(t,L), \eta_{xx}(t-\tau(t),L)) \right\rangle,
\end{align*}with (recall \eqref{matrix_nega_def_albe})
\begin{align*}
\Psi_{\mu_1,\mu_2} =
	\Phi_{\alpha,\beta}
	+ \frac{a_1 L\mu_1 }{2}\begin{pmatrix}
		\alpha^2+1 & \alpha\beta \\
		\alpha\beta & \beta^2
	\end{pmatrix}
	+ \frac{\lvert\beta\rvert\mu_2}{2} \begin{pmatrix}
		1 & 0 \\ 0 & 0
	\end{pmatrix},
\end{align*}
\begin{align*}
S_2 =& -\frac{\mu_1}{2} \int_{0}^{L}\left(\omega^2+\eta^2\right)dx+\frac{3 a \mu_1}{2}\int_{0}^{L}\left(\omega_{x}^{2}+\eta_{x}^{2}\right)dx+\frac{\lambda}{2} \int_0^L\left(\eta^2+\omega^2\right) dx  \\
& +\mu_1\lambda\int_{0}^{L} x\eta \omega dx -\frac{ 5 a_1 \mu_1}{2}\int_{0}^{L}\left(\omega_{xx}^{2}+\eta_{xx}^{2}\right)dx,
\end{align*}and
\begin{equation*}
				\begin{aligned}
		S_3=&-\mu_2\frac{\lvert\beta\rvert}{2}\int_0^1(1-\dot\tau(t)\rho)\eta_{xx}^2(t-\tau(t)\rho,L) d\rho + \frac{\lambda \lvert \beta\rvert }{2}\tau(t) \int_0^1 \eta_{xx}^2(t-\tau(t)\rho , L) d\rho\\
		& +\frac{\mu_2\lvert\beta\rvert{ \lambda}}{2} \tau(t) \int_0^1 (1-\rho) \eta_{xx}^2(t-\tau(t)\rho ,L) d\rho,
	\end{aligned}
\end{equation*}respectively.

The objective is to show that $V'(t)+\lambda V(t)\leq0$. To do so, let us analyze each term $S_i$ in \eqref{DV}, for $i=1,2,3$.

\vspace{0.2cm}

\noindent\textbf{Estimate of $S_1$:} Since the matrix  $\Phi_{\alpha,\beta}$ (see \eqref{matrix_nega_def_albe}) is definite negative, it follows from the continuity of the trace and determinant functions that one choose $\mu_1,\mu_2\in(0,1)$ sufficiently small so that the new matrix $\Psi_{\mu_1,\mu_2}$ is also negative definite. Thus,
$$
	S_1 \leq0.
$$

\vspace{0.2cm}

\noindent\textbf{Estimate of $S_2$:} Observe that using Poincar\'e inequality, we get
\begin{equation*}
	\begin{aligned}
		S_2
		\leq & \frac{L^2}{2 \pi^2}\lambda(1+\mu_1L) \int_{0}^{L}\left( \omega_x^2+\eta_x^2 \right)dx +\frac{3 a \mu_1}{2}\int_{0}^{L}\left(\omega_{x}^{2}+\eta_{x}^{2}\right)dx\\
  &  -\frac{5 a_1 \mu_1}{2}\int_{0}^{L}\left(\omega_{xx}^{2}+\eta_{xx}^{2}\right)dx\\
		\leq& \left[\frac{L^2}{2\pi^2}\left(\lambda\left(1+\mu_1L\right)\frac{L^2}{2\pi^2}+3 a \mu_1 \right)-\frac{5 a_1\mu_1}{2}\right]\int_{0}^{L}\left(\omega_{xx}^{2}+\eta_{xx}^{2}\right)dx.
	\end{aligned}
\end{equation*}
Thus,
\begin{equation*}
	S_2<0,
\end{equation*}
if
\begin{equation*}
	\lambda < \frac{\mu_1 \pi^2\left(5 a_1\pi^2-3 a L^2\right)}{L^4(1+\mu_1)}.
\end{equation*}
\vspace{0.2cm}

\noindent\textbf{Estimate of $S_3$:}  We proceed as in \cite{Boumediene2024}, choosing
\begin{equation*}
	\lambda < \frac{\mu_2(1-d)}{M(1+\mu_2)},
\end{equation*} it follows that
\begin{equation*}
S_3 < 0.
\end{equation*}Therefore, for the estimates above, we have
$$
\frac{d}{dt} V(t) + \lambda V(t) \leq 0,
$$ and,  since $V(t)$ satisfies \eqref{eq:EquivEV},  we deduce that
\begin{align*}
E(t) \leq  \zeta E(0)e^{-\lambda t}, \quad \forall t\geq 0,
\end{align*}for $\zeta>0$ and $\lambda> 0$ fulfilling \eqref{eq:r} and \eqref{eq:lambda}, respectively. This achieves the proof of the theorem.\qed

\subsection{Decay rate: an optimal result}
We can optimize the value of $\lambda$ in Theorem \ref{th:Lyapunov0} to obtain the best decay rate for the linear system \eqref{eq:KdV-KdVaaaa} in the following way:
\begin{proposition}\label{llllll} If the constant $\mu_1$ giving in Theorem \ref{th:Lyapunov0} is chosen as follows
\begin{equation}\label{eq:muOPT}
	\mu_1 \in \left[0,\frac{(2 a_1\alpha-\lvert\beta\rvert)(1-d)-a_1^2\lvert\beta\rvert}{L(1-d)(a_1^2+\alpha^2)} \right),
\end{equation}
then we can have that $\lambda$ is the largest possible.
\end{proposition}
\begin{proof}
Define the functions $f$ and $g$ $\colon \left[
	0, \frac{(2a_1\alpha-\lvert\beta\rvert)(1-d)-a_1^2\lvert\beta\rvert}{L(1-d)(a_1^2+\alpha^2)}
	\right]\to \mathbb{R}$
by
$$		f(\mu_1) = \frac{\mu_1 \pi^2\left(5a_1\pi^2-3aL^2\right)}{L^4(1+\mu_1L)},$$
and
$$	g(\mu_1) = \frac{(2a_1\alpha-|\beta|)(1-d)-a_1^2|\beta|-L(1-d)(a_1^2+\alpha^2)\mu_1}{M\left(2a_1\alpha(1-d)-a_1^2|\beta|-L(1-d)(a_1^2+\alpha^2)\mu_1\right)}(1-d),$$ respectively. On the other hand, let us consider $\lambda(\mu_1) = \min\lbrace f(\mu_1), g(\mu_1)\rbrace$. Thus,  we have the following claims.

\begin{claim}\label{RR1} The function $f$ (resp. $g$) is increasing (resp. decreasing) in the interval $$\left[0,\dfrac{(2a_1\alpha-\lvert\beta\rvert)(1-d)
-a_1^2\lvert\beta\rvert}{L(1-d)(a^2_1+\alpha^2)}\right).$$
\end{claim}

\noindent A simple computation shows that
$$ f^{\prime}(\mu_1)>0, \, \, \text{for all }\, \,  \mu_1 \geq 0$$ and hence  $f^{\prime}(\mu_1)>0$ for $\mu_1 \in\left[0,\frac{(2a_1\alpha-\lvert\beta\rvert)(1-d)-a^2_1\lvert\beta\rvert}{L(1-d)(a_1^2+\alpha^2)}\right)$.

\noindent Furthermore, one can rewrite $g$ as follows
$$g(\mu_1)=\dfrac{1-d}{M}-\frac{|\beta|(1-d)^2}{ML(1-d)(a_1^2+\alpha^2)}
\left(\dfrac{1}{\dfrac{2a_1\alpha(1-d)-a^2_1|\beta|}{L(1-d)(a_1^2+\alpha^2)}
-\mu_1}\right),$$ and thus
$$
	g^{\prime}(\mu_1)=-\dfrac{|\beta|(1-d)^2}{ML(1-d)(a_1^2+\alpha^2)}
	\left[
	\dfrac{1}{
		\left(
		\dfrac{2a_1\alpha(1-d)-a_1^2|\beta|}{L(1-d)(a_1^2+\alpha^2)}-\mu_1
		\right)^2
	}
	\right]<0.
$$ This ascertains the claim \ref{RR1}.

\begin{claim}\label{RR2}
There exists only one point $\mu_1$, satisfying~\eqref{eq:muOPT} such that $f(\mu_1)=g(\mu_1)$.
\end{claim}

\noindent Indeed, since
\begin{align*}
f(0)=0, \qquad   f\left(\dfrac{(2 a_1\alpha-\lvert\beta\rvert)(1-d)-a_1^2\lvert\beta\rvert}{L(1-d)(a_1^2+\alpha^2)}\right) >0,
\end{align*}and
\begin{equation*}
g(0)>0, \qquad g\left(\dfrac{(2a_1\alpha-\lvert\beta\rvert)(1-d)-a_1^2\lvert\beta\rvert}
{L(1-d)(a_1^2+\alpha^2)}\right)=0,
\end{equation*} the existence of this point is a direct consequence of the Mean Value Theorem, applied to the function $F=f-g$. The uniqueness follows from the fact that the function $F = f-g$ is increasing in this interval, and claim \ref{RR2} holds.

Lastly, thanks to the claims \ref{RR1} and \ref{RR2}, the maximum value of the function $\lambda$ is obtained when $\mu_1$ satisfying~\eqref{eq:muOPT}, where $f(\mu_1)=g(\mu_1)$, and the proof of Proposition \ref{llllll} is achieved.
\end{proof}

\section{Conclusion} \label{sec4}
This paper establishes the existence and uniqueness of a solution for a higher-order nonlinear Boussinesq system in a bounded domain, even when a time-dependent delay is present in one of the boundary conditions. Additionally, we prove that solutions to the linearized problem are exponentially stable, both results being obtained under certain conditions related to the system’s parameters and the delay. These findings extend the results of the second and third authors in \cite{Boumediene2024} for a higher-order dispersive system. Further comments on our results are provided below.
\begin{enumerate}
\item  It is worth mentioning that the solutions of the system \eqref{eq:KdV-KdV} obtained in Theorem \ref{nonlinearT} are local. Proving the global existence of solutions remains a challenge due to the absence of an a priori $L^2$-estimate. Specifically, it is challenging to tackle this problem within the energy space for the nonlinear system that includes a delay term.
\item  Observe that the restriction $0 < L < \sqrt{\frac{5 a_1}{3 a}}\pi$ in Theorem \ref{th:Lyapunov0} arises from the Kato smoothing effect, which does not occur in the lower-order Boussinesq system (see, for example, \cite{Boumediene2024}). This difference is because in system \eqref{eq:KdV-KdV}, we have spatial derivatives of orders three and five, both with positive signs. Thus, after performing some integration by parts, the left-hand side of \eqref{mul_part_1} contains the $H^1$-norm with a negative sign and the $H^2$-norm with a positive sign. To recover the $H^2$-norm, the Poincaré inequality must be applied, which imposes this restriction on the size of $L$.
\item A version of the higher-order Boussinesq system was proposed by \cite[equations (4.7) and (4.8), p. 283]{Olver} and is given by:
$$
\begin{cases}
\eta_t+u_x+\frac{1}{6} \beta\left(3 \theta^2-1\right) u_{x x x}+  \frac{1}{120} \beta^2\left(25 \theta^4-10 \theta^2+1\right) u_{x x x x x}\vspace{0.2cm} \\
\quad\quad\quad\quad\quad\quad\quad\quad\quad\quad\quad\quad\quad\quad\quad\quad\quad\quad\quad+\alpha(\eta u)_x+\frac{1}{2} \alpha \beta\left(\theta^2-1\right)\left(\eta u_{x x}\right)_x=0, \vspace{0.2cm}\\
u_t+\eta_x+\beta\left[\frac{1}{2}\left(1-\theta^2\right)-\tau\right] \eta_{x x x}+\beta^2\left[\frac{1}{24}\left(\theta^4-6 \theta^2+5\right)+\frac{\tau}{2}\left(\theta^2-1\right)\right] \eta_{x x x x x}\vspace{0.2cm} \\
\quad\quad\quad\quad\quad\quad\quad\quad\quad\quad\quad\quad\quad\quad\quad\quad+\alpha u u_x+\alpha \beta\left[\left(\eta \eta_{x x}\right)_x+\left(2-\theta^2\right) u_x u_{x x}\right]=0.
\end{cases}
$$
Through scaling, we arrive at the following system:
\begin{equation}\label{cagapa}
\begin{cases}\eta_t+u_x-a u_{x x x}+a_1(\eta u)_x+a_2\left(\eta u_{x x}\right)_x+b u_{x x x x x}=0, & \text { in }(0, L) \times(0, \infty), \\ u_t+\eta_x-a \eta_{x x x}+a_1 u u_x+a_3\left(\eta \eta_{x x}\right)_x+a_4 u_x u_{x x}+b \eta_{x x x x x}=0, & \text { in }(0, L) \times(0, \infty), \\ \eta(x, 0)=\eta_0(x), \quad u(x, 0)=u_0(x), & \text { in }(0, L),
\end{cases}
\end{equation}
where $a>0, b>0, a \neq b, a_1>0, a_2<0, a_3>0$ and $a_4>0$. The system \eqref{cagapa} was studied in \cite{CF2019}.  Using the same boundary conditions as in the problem \eqref{eq:KdV-KdV}, we believe that similar results proved in our work can be obtained for the system \eqref{cagapa} without the restriction over $L$ since the signal of the third derivatives in \eqref{cagapa} is negative instead of positive as in our case, see system \eqref{eq:KdV-KdV}.
\item It is important to point out that the system \eqref{eq:KdV-KdV} is shown to be locally well-posed, and hence we are unable to establish any exponential stability for the nonlinear problem. One interesting research avenue is to show a stability outcome for the nonlinear problem.
\end{enumerate}

\subsection*{Acknowledgment} The authors are grateful to the associate editor and the referees for the careful reading of this paper and their valuable suggestions and comments. Capistrano–Filho was partially supported by CAPES/COFECUB grant number 88887.879175/2023-00, CNPq grant numbers 421573/2023-6 and 307808/2021-1, and Propesqi (UFPE). Bautista  is supported by Universidad Tecnol\'ogica de los Andes, Abancay-Peru. Sierra is supported by Escola de Matem\'atica Aplicada, Funda\c{c}\~{a}o Get\'ulio Vargas, Rio de Janeiro-Brazil.
\subsection*{Data Availability} It does not apply to this article as no new data were created or analyzed in this study.
\subsection*{Conflict of interest} This work does not have any conflicts of interest.

\end{document}